\documentclass[11pt, oneside]{jams-l}
\usepackage[T1]{fontenc}
\usepackage{mathrsfs}
\DeclareMathOperator{\lcm}{\operatorname{lcm}}

\usepackage{tikz-cd}
\usepackage{tikz}
\usepackage{lipsum}
\usepackage{adjustbox}

\usepackage{amssymb}

\usepackage{graphicx}

\usepackage[cmtip,all]{xy}

\newtheorem{theorem}{Theorem}[section]
\newtheorem{lemma}[theorem]{Lemma}
\newtheorem{proposition}[theorem]{Proposition}
\newtheorem{corollary}[theorem]{Corollary}

\theoremstyle{definition}
\newtheorem{definition}[theorem]{Definition}

\theoremstyle{remark}

\newtheorem{question}[theorem]{Question}

\numberwithin{equation}{section}

\begin{document}

\title{THE PROBABILISTIC ZETA FUNCTION OF A FINITE LATTICE}


\author{Besfort Shala}
\address{}
\curraddr{}
\email{}



\date{}

\dedicatory{}

\begin{abstract} We study Brown's definition of the probabilistic zeta function of a finite lattice as a generalization of that of a finite group. We propose a natural alternative or extension that may be better suited for non-atomistic lattices. The probabilistic zeta function admits a general Dirichlet series expression, which unlike for groups, need not be ordinary. We compute the function for several examples of finite lattices, establishing a connection with the Stirling numbers of the second kind in the case of the divisibility lattice. Furthermore, in the context of moving from groups to lattices, we are interested in lattices with probabilistic zeta function given by ordinary Dirichlet series. In this regard, we focus on partition lattices and $d$-divisible partition lattices. Using the prime number theorem, we show that the probabilistic zeta functions of the latter typically fail to be ordinary Dirichlet series.
\end{abstract}

\maketitle

\section{INTRODUCTION}

Let $G$ be a finite group and $s$ a non-negative integer. Define $P(G, s)$ to be the probability that a randomly chosen $s$-tuple from $G^s$ generates $G$. Hall \cite{hall} essentially found (in slightly different terminology) a finite ordinary Dirichlet series expression for $P(G, s)$, which may be used to extend the definition of $P(G, s)$ to the complex plane. The \textbf{probabilistic zeta function of $G$} is then defined to be the reciprocal of the complex function $P(G, s)$. The name is motivated by the probabilistic interpretation of $1/\zeta(2)$, where $\zeta$ is the Riemann zeta function (see \cite[Theorem 332]{hardywright}). Moreover, the Euler product identity for $\zeta$ gives $$\frac1{\zeta(s)} = \prod_p P(\mathbb Z/p\mathbb Z, s),$$ where $\mathbb Z/p\mathbb Z$ is the cyclic group of prime order $p$. Brown \cite{brown} defined an analogous probabilistic zeta function $1/P(L, s)$ for finite lattices in order to show that $P(G, s)$ depends only on the coset lattice of $G$, that is, a lattice structure associated to $G$. The aim of this paper is to study this definition for lattices in its own right.

We present some details of the derivation of the Dirichlet series for $P(G, s)$ for completeness. Clearly, $P(G,s)\cdot|G|^s$ is the number of $s$-tuples which generate $G$. Since any $s$-tuple in $G^s$ generates some subgroup $H\leq G$, we have \begin{equation}|G|^s = \sum_{H\leq G} P(H, s)\cdot|H|^s.\label{1}\end{equation} We may extract $P(G, s)$ from \eqref{1} via M\"obius inversion for an arbitrary finite partially ordered set (from here on, \textbf{poset}) $\mathscr P$ with partial order $\leq$, a technique first introduced by Hall \cite{hall}. For us, the following ``basic'' description of the procedure will suffice. We would like to have a real valued function $\mu$ defined on $\mathscr P\times\mathscr P$, such that for any two real valued functions $f$ and $g$ defined on $\mathscr P$ with $f(x) = \sum_{y\leq x} g(y)$, we have $$g(x) = \sum_{y\leq x} \mu(y, x)f(y).$$ Plugging in the relation for $f$ in terms of $g$, one easily arrives to the recursive definition $\mu(x, x) = 1$ for all $x\in\mathscr P$ and $\mu(y, x) = -\sum_{y< z\leq x} \mu(z, x)$ for $y<x$.

Using M\"obius inversion on the poset of subgroups of $G$, as well as noticing that \eqref{1} holds for any subgroup $H$ of $G$ (since it holds for an arbitrary group), we obtain \begin{equation} P(G, s) = \sum_{H\leq G} \frac{\mu(H, G)}{[G:H]^s} = \sum_{n=1}^{\infty} \frac{\sum_{H\leq G;\text{ } [G:H] = n}\mu(H, G)}{n^s}\label{3},
\end{equation} where $\mu(H, G)$ are the M\"obius numbers of the poset of subgroups of $G$. The Dirichlet series on the right-hand side of \eqref{3} is finite, has integer coefficients and may be used to extend the domain of $P(G, s)$ to the entire complex plane $\mathbb C$.

Brown \cite{brown} uses a very similar setup and argument to define the probabilistic zeta function of a finite lattice. He showed that the probabilistic zeta function of a finite group $G$ is related to that of its coset lattice $\mathscr C(G)$ via $P(\mathscr C(G), s+1) = P(G, s)$.

In this paper, we give a natural alternative or extension of the probabilistic zeta function of a finite lattice, by using join-irreducible elements as building blocks. The main novelty is that the probabilistic zeta function of a finite lattice, unlike that of a group, need not be an ordinary Dirichlet series. In this context, we begin a study of this lattice invariant in its own right. For example, the probabilistic zeta function on the lattice of flats of a matroid (so, any geometric lattice) probabilistically counts the generating sets for the matroid. In particular, we are interested in lattices which retain the ordinary Dirichlet series structure. This motivates us to define ``strongly'' and ``weakly'' coset-like lattices, both of which have ordinary Dirichlet series expressions for their probabilistic zeta functions. We then compute examples explicitly and investigate coset-like behaviour. In doing so for the $d$-divisible partition lattice, we use a consequence of the prime number theorem on the existence of primes in intervals of the form $[x, (1+\varepsilon)x]$ to show that the $d$-divisible partition lattice is typically not ``strongly'' coset-like.

\subsection*{Structure of the paper} This paper is structured as follows. Section 2 contains the basic definitions and theorems we shall need throughout the paper. In Section 3, we propose a natural alternative definition of the probabilistic zeta function $P(L, s)$ for finite lattices, which may be better-suited for non-atomistic lattices. For atomistic lattices, our definition is identical to the one given by Brown \cite{brown}.

Section 4 contains computations of $P(L, s)$ for a number of examples of finite lattices as well as some connections with well-known identities. For example, in the case of the divisibility lattice of a square-free integer (i.e. the Boolean lattice), we recover the Stirling numbers of the second kind. This gives a natural extension of the Stirling numbers when considering the non-restricted divisibility lattice. In Section 5, we define coset-like lattices and investigate examples further. 

Section 6 contains a discussion on the compatibility of the the probabilistic zeta function with the lower reduced product of lattices. Finally, we present some questions that remain open in Section 7.

\subsection*{Acknowledgements} I would like to thank my mentor, Professor Russ Woodroofe, for proposing this research project as well as for his continuous support and helpful suggestions throughout our work together. A very special thanks goes to Håvard Damm-Johnsen for setting up and running a computer program to aid an example searching process. I also express my gratitude towards Milo\v s Puđa for all the conversations and help in computations.

\section{PRELIMINARIES}

We use following definition of a lattice.

\begin{definition}
Let $L$ be a non-empty set with two commutative, associative and idempotent binary operations $\vee :  L\times L\to L$ (the \textbf{join}) and $\wedge :  L\times L \to L$ (the \textbf{meet}) such that for all $x, y\in L$, the meet and the join are related in the following way: $$x\vee(x\wedge y) = x \text{ and } x\wedge (x\vee y) = x.$$ The algebraic structure $( L, \vee, \wedge)$ is said to be a \textbf{lattice}.   
\end{definition} 

The above definition is equivalent to the following one from order theory (see \cite[Chap. I, Sect. 1]{gratzer} for details) and we shall use them interchangeably as deemed appropriate.

\begin{proposition}
Given a lattice $L$, we may define a partial order $\leq$ on $L$ given by $x\leq y$ when $x\vee y = y$. Moreover, any two elements of $L$ have a unique supremum and a unique infimum with respect to $\leq$, which coincide with the join and the meet of the two elements, respectively. Conversely, we can construct a lattice from any poset with this property.
\end{proposition}

As we are only interested in finite lattices, from here on, the word `lattice' should be interpreted as `finite lattice'. 

It is well-known that a lattice has unique identities for both of its binary operations. We write $\widehat 0$ (and call it the \textbf{bottom} element) and $\widehat 1$ (the \textbf{top} element) for the identities of $ L$ with respect to $\vee$ and $\wedge$, respectively. We say an element $x\in L\setminus\{\widehat 0\}$ is \textbf{join-irreducible} if it cannot be written as a non-trivial join of elements of $ L$. Equivalently (due to finiteness), an element $x$ is join-irreducible if $x = a\vee b$ implies that $a = x$ or $b = x$. Furthermore, given an element $x\in L\setminus\{\widehat 0\}$, it is either join-irreducible or else may be expressed as a non-trivial join, say $a\vee b$. Repeating the same argument for $a$ and $b$ and continuing in this manner, due to finiteness, we obtain a ``factorization'' of $x$ as a join of join-irreducible elements of $L$. In this sense, join-irreducible elements serve as building blocks for lattices.

Next, we lay out the theorems that we shall need from number theory.

\begin{theorem}[The Prime Number Theorem \cite{apostol}]
For a real number $x$, let $\pi(x)$ be the number of prime numbers less than or equal to $x$. Then, $$\lim_{x\to\infty} \frac{\pi(x)}{\frac{x}{\log x}} = 1.$$ In particular, for any given $\varepsilon > 0$, there exists $N=N(\varepsilon)$ such that for all $x\geq N$, there is a prime number between $x$ and $(1+\varepsilon)x$. \label{pnt}
\end{theorem}

In the context of the above statement of the prime number theorem and its consequence, the following theorem of Nagura \cite{nagura} gives an explicit value of $N(\varepsilon = 1/5)$.

\begin{theorem} There exists a prime between $n$ and $6n/5$ for any $n\geq 25$.\label{nagurathm}\end{theorem}

\section{THE PROBABILISTIC ZETA FUNCTION OF A FINITE LATTICE}

Let $ L$ be a lattice with distinct identities $\widehat 0$ and $\widehat 1$. Furthermore, let $J = J(L)$ be the set of join-irreducible elements in $ L$. We say an $s$-tuple $(x_1, \ldots, x_s)$ in $J^s$ \textbf{generates up to} $x\in L\setminus\{\widehat 0\}$ if $\bigvee_{i=1}^s x_i = x$. In this case, $x_i\leq x$ for each $i$, therefore all components of such $s$-tuples are elements of $$J_x := \{j\in J : j\leq x\}.$$ If $(x_1, \ldots, x_s)$ generates up to $\widehat 1$, we say the $s$-tuple \textbf{generates} $ L$. Let $P( L, x, s)$ be the probability that a randomly chosen $s$-tuple from $J_x^s$ generates up to $x$ and define $P( L, s) := P( L, \widehat 1, s)$. Clearly, $P( L, x, s)|J_x|^s$ is the number of $s$-tuples in $J_x^s$ which generate up to $x$. Since every $s$-tuple in $J_x^s$ generates up to some $\widehat 0 < y\leq x$, we have $$|J_x|^s = \sum_{\widehat 0 < y\leq x} P( L, y, s)|J_y|^s.$$ Applying M\"obius inversion on $ L\setminus\{\widehat 0\}$ as the underlying poset and plugging in $x=\widehat 1$, we obtain \begin{equation}P( L, s) = \sum_{x\in L\setminus\{\widehat 0\}} \frac{\mu(x, \widehat 1)}{[J : J_x]^s},\label{4}\end{equation} where $[J : J_x] := |J|/|J_x|$ is the ratio of the number of join-irreducibles in $L$ and the number of join-irreducibles less than or equal to $x$ (notice that $J_{\widehat 1} = J$). We note here that if $x<\widehat 1$ is maximal in $L$, i.e. a \textbf{coatom}, then $\mu(x, \widehat 1) = -1$ by definition of the M\"obius numbers. As before, the expression on the right-hand side could be used to extend the domain of $P( L, s)$ to the entire complex plane $\mathbb C$. We refer to $P(L, s)$ as the \textbf{probabilistic zeta function of $L$}. We remark that this is not consistent with the definition of the probabilistic zeta function of a finite group, which is defined to be $1/P(G, s)$, but we choose convenience over consistency.

Some important remarks are due here. The definition of $P(L, s)$ found in the last section of \cite{brown} differs slightly from the one we have just given. This is because Brown sets $J$ to be the set of minimal elements of $L\setminus\{\widehat 0\}$, namely the \textbf{atoms} of $L$. However, we may well be in a situation where the join of all atoms is not $\widehat 1$, and consequently where none of the $s$-tuples would generate the whole lattice. Any chain of length at least two is an example of this situation; a more interesting example is the divisibility lattice of any non square-free integer. We avoid this degeneracy via our definition, since the join of all join-irreducible elements below a given element $x$ of a lattice is equal to $x$ (so this holds for $x=\widehat 1$ in particular). Indeed, note that the two definitions are equivalent for \textbf{atomistic} lattices, namely lattices where the join-irreducible elements are precisely the atoms.

Brown \cite{brown} related the probabilistic zeta function of a finite group to that of an atomistic lattice associated to the group (in fact, this is the primary reason why Brown introduced the concept of the probabilistic zeta function for lattices). For a finite group $G$, let $\mathscr C(G)$ be the set of all cosets of all subgroups of $G$, together with the empty set. Then, $\mathscr C(G)$ is a lattice with meet given by set intersection, and join given by $x_1H_1\vee x_2H_2 = x_1H = x_2H$, where $H = \langle x_1^{-1}x_2, H_1, H_2 \rangle$. We shall refer to $\mathscr C(G)$ as the \textbf{coset lattice of $G$}. Indeed, an element of $\mathscr C(G)$ is join-irreducible if and only if it is a coset of the identity group, so $J(\mathscr C(G))$ corresponds to $G$. Brown proved that \begin{equation}P(\mathscr C(G), s+1) = P(G, s)\label{eq2}\end{equation} by noticing that an $(s+1)$-tuple $(x_0, x_1, \ldots, x_s)$ generates $\mathscr C(G)$ if and only if the $s$-tuple $(x_0^{-1}x_1, x_0^{-1}x_2, \ldots, x_0^{-1}x_s)$ generates $G$. Of course, $\mathscr C(G)$ is atomistic, hence \eqref{eq2} holds also for our definition of the probabilistic zeta function.

\section{COMPUTATION OF EXAMPLES}

In this section, we compute the probabilistic zeta function on a number of examples of lattices, establishing connections with well-known identities. 

\subsection{Divisibility Lattice}

Let $\mathcal O_n = \{d\in\mathbb N : d\mid n\}$ be the set of positive divisors of a positive integer $n>1$ with canonical factorization $p_1^{\alpha_1}p_2^{\alpha_2}\cdots p_r^{\alpha_r}$, where $p_1, \ldots, p_r$ are distinct primes and $\alpha_i\geq 1$ for all $i=1, 2, \ldots, r$. Indeed, $\mathcal O_n$ is a lattice with join $\lcm$ and meet $\gcd$. The join-irreducible elements of $\mathcal O_n$ are precisely the prime powers $p_i, p_i^2, \ldots, p_i^{\alpha_i}$ for each $i=1, 2, \ldots, r$. Thus, there are $\sum_1^r \alpha_i$ join-irreducible elements in total. For a given $d\in\mathcal O_n$, we may write $d=p_1^{\beta_1}p_2^{\beta_2}\cdots p_r^{\beta_r}$, where $0\leq\beta_i\leq \alpha_i$ for all $i=1, 2, \ldots, r$, so the number of join-irreducible elements which are at most $d$ is $\Omega_d = \sum_1^r \beta_i$ (in particular, $\Omega_n = \sum_1^r \alpha_i$). The M\"obius numbers of the divisibility lattice (as a poset) are closely related to the usual number-theoretic M\"obius function. This is the content of the following lemma from \cite[Chap. 3, Sect. 8]{stanley}.

\begin{proposition}
The M\"obius numbers $\mu(d, n)$ of the divisibility lattice $\mathcal O_n$ are given by $\mu(d, n) = \mu(n/d)$, where the $\mu$ on the right-hand side is the usual number-theoretic M\"obius function.
\end{proposition}

As an immediate corollary, $\mu(d, n)$ is non-zero if and only if $n/d$ is square-free. In this case, we have $\mu(d, n) = \mu(n/d) = (-1)^{\Omega_n - \Omega_d}$ and $\Omega_d\geq\sum_{i=1}^r (\alpha_i - 1) = \Omega_n - r$. 

Putting everything together, we obtain \begin{align} P(\mathcal O_n, s) & =  \sum_{1<d\mid n} \frac{\mu(d, n)}{[J: J_d]^s}  =  \sum_{\substack{1<d\mid n; \nonumber}} \mu^2\left(\frac{n}{d}\right)\frac{(-1)^{\Omega_n - \Omega_d}}{(\Omega_n/\Omega_d)^s} \\  & =  \frac{1}{\Omega_n^s}\sum_{k=\Omega_n-r}^{\Omega_n}(-1)^{\Omega_n - k}\sum_{\substack{1<d\mid n;\\\text{ }\Omega_d = k}}\mu^2\left(\frac nd\right)k^s \nonumber\end{align}\begin{align} &= \frac{(-1)^{\Omega_n}}{\Omega_n^s} \sum_{k=\Omega_n-r}^{\Omega_n}(-1)^{k} \binom{r}{\Omega_n-k}  k^s. \label{6}
\end{align} The last equality follows as there are precisely $\binom{r}{\Omega_n - k}$ divisors $d$ of $n$ such that $n/d$ is square-free and $\Omega_d = k$ (from the $r$ primes, we may choose any $\Omega_n-k$ of them to set their exponent equal to $\alpha_i - 1$). 

Let us now consider the special case when $\Omega_n = r$, i.e. when $n$ is square-free. The expression \eqref{6} simplifies to \begin{equation}P(\mathcal O_{p_1\cdots p_r}, s) = \frac{(-1)^r}{r^s}\sum_{k=1}^r (-1)^k\binom{r}{k}k^s.\label{7}\end{equation} Firstly, since $\mathcal O_{p_1\cdots p_r}$ and the Boolean lattice $\mathcal B_r$ of all subsets of $\{1, 2, \ldots, r\}$ ordered by inclusion are isomorphic, \eqref{7} is also the probabilistic zeta function of $\mathcal B_r$.

Secondly, notice that for positive integer values of $s$, \eqref{7} may be rewritten as $r!S(s, r)/r^s$ where $S(s, r)$ is the Stirling number of the second kind (see \cite[1.94a]{stanley}), namely the number of ways of partitioning $s$ elements into $r$ non-empty parts. Of course, an $s$-tuple generates $\mathcal O_{p_1\cdots p_r}$ if and only if it contains all of the primes $p_1, \ldots, p_r$ and for $r$ given primes, there are precisely $r!S(s, r)$ ways of distributing the positions $1, 2, \ldots, s$ of the $s$-tuple to the $r$ ordered parts determined by the primes. In other words, $r!S(s, r)$ is the number of $s$-tuples in $\{p_1, \ldots, p_r\}^s$ which contain all of $p_1, \ldots, p_r$. We summarize this in the following corollary. 

\begin{corollary}
For positive integer values of $s$, we have $$P(\mathcal B_r, s) = \frac{r!}{r^s}S(s, r),$$ where $\mathcal B_r$ is the Boolean lattice and $S(s, r)$ is the Stirling number of the second kind.
\end{corollary}

Returning to \eqref{6} and noticing that an $s$-tuple generates $\mathcal O_n$ if and only if it contains all of the $r$ maximal prime powers $p_1^{\alpha_1}, \ldots, p_r^{\alpha_r}$, we get that \eqref{6} is a possible generalization of the Stirling numbers of the second kind, in the sense that $\Omega_n^s P(\mathcal O_n, s)$ is the number of $s$-tuples over a set with $\Omega_n$ elements which contain all elements of a fixed subset with $r$ elements. Generalizations of the Stirling numbers of the second kind have been studied in \cite{griffiths, afrika}.

\subsection{Subspace Lattice of a Finite-Dimensional Vector Space Over a Finite Field}

Let $q$ be a prime power and consider the set $\mathcal S(\mathbb F_q^n)$ of all subspaces of the vector space $\mathbb F_q^n$ over the field with $q$ elements $\mathbb F_q$. Indeed, $\mathcal S(\mathbb F_q^n)$ is a lattice with join $+$ (i.e. addition of vector subspaces) and meet $\cap$. The join-irreducible elements of $\mathcal S(\mathbb F_q^n)$ are precisely its atoms, that is, the $1$-dimensional subspaces of $\mathbb F_q^n$. Thus, the number of join-irreducible elements in $\mathcal S(\mathbb F_q^n)$ is $(q^n - 1)/(q-1)$. 

Now, let $V\leq\mathbb F_q^n$ be a vector subspace. Notice that $|J_V|$, the number of join-irreducible elements which are at most $V$, that is to say, the number of $1$-dimensional vector subspaces of $V$, is precisely $(q^{\dim V} - 1)/(q-1)$. The M\"obius numbers were found by Hall \cite{hall}. \begin{proposition} The M\"obius numbers $\mu(V, \mathbb F_q^n)$ of the subspace lattice $\mathcal S(\mathbb F_q^n)$ are given by $\mu(V, \mathbb F_q^n) = (-1)^{n-\dim V}q^{\binom{n-\dim V}{2}}$. \end{proposition} Thus, we immediately obtain \begin{eqnarray*}P(\mathcal S(\mathbb F_q^n), s) & = & \frac{1}{(q^n-1)^s} \sum_{0\neq V\leq\mathbb F_q^n} (-1)^{n-\dim V}q^{\binom{n-\dim V}{2}}(q^{\dim V}-1)^s\\ & = & \frac{1}{(q^n - 1)^s}\sum_{k=1}^n (-1)^{n-k}\left[\begin{matrix} n \\ k \end{matrix}\right]_q q^{\binom{n-k}{2}} (q^k - 1)^s,\end{eqnarray*} where $\left[\begin{matrix} n \\ k \end{matrix}\right]_q$ is the number of $k$-dimensional subspaces of $\mathbb F_q^n$. We recall the explicit expression (see \cite[Chap 1, Sect. 7]{stanley}) $$\left[\begin{matrix} n \\ k \end{matrix}\right]_q = \frac{(1-q^n)(1-q^{n-1})\cdots(1-q^{n-k+1})}{(1-q^k)(1-q^{k-1})\cdots(1-q)},$$ which is a ``$q$-analog'' of the binomial coefficient.

The subspace lattice $\mathcal S(\mathbb F_q^n)$ is a natural extension of the Boolean lattice $\mathcal B_n$, in the sense that many results about $\mathcal S(\mathbb F_q^n)$ degenerate to results about the Boolean lattice $\mathcal B_n$ when $q\to 1$. Thus, in many ways $\mathcal B_n$ plays the role of the subspace lattice of a $n$-dimensional vector space over a field with one element (if one were to exist). The following proposition is yet another instance of this. 

\begin{proposition}
Viewing $P(\mathcal S(\mathbb F_q^n), s)$ as a continuous function of $q$ with a removable singularity at $q=1$, we have  $\lim_{q\to 1} P(\mathcal S(\mathbb F_q^n), s) = P(\mathcal B_n, s)$. 
\end{proposition}

\begin{proof}
This immediately follows by the well-known identity $\lim_{q\to 1} \frac{q^i-1}{q^j - 1} = \frac{i}{j}$ which gives $\lim_{q\to 1} \left[\begin{matrix}n \\ k \end{matrix} \right]_q = \displaystyle\binom{n}{k}$. Now, $$\lim_{q\to 1} \frac{1}{(q^n - 1)^s}\sum_{k=1}^n (-1)^{n-k}\left[\begin{matrix} n \\ k \end{matrix}\right]_q q^{\binom{n-k}{2}} (q^k - 1)^s = \frac{(-1)^n}{n^s}\sum_{k=1}^n (-1)^{k}\binom{n}{k}k^s = P(\mathcal B_n, s),$$ as desired. 
\end{proof}

\subsection{Partition Lattice}

Let $\Pi_n$ be the set of all partitions of $\{1, 2, \ldots, n\}$, ordered by refinement. Then, $\Pi_n$ is a lattice (see \cite[Chap. 3, Sect. 10]{stanley}). The join-irreducible elements of $\Pi_n$ are precisely its atoms, i.e. partitions with only one non-trivial part of size $2$. As such, we have $|J| = \binom{n}{2}$, for the atoms of $\Pi_n$ bijectively correspond to the $2$-element subsets of $\{1, 2, \ldots, n\}$. The join-irreducible elements which are at most a partition $P$ are precisely the partitions corresponding to all possible pairs of elements in the parts $P_1, \ldots, P_k$ of $P$, therefore $J_P = \bigsqcup_{i=1}^k \binom{P_i}{2}$. This gives $|J_P| = \sum_{i=1}^k \binom{|P_i|}{2}$. The relevant M\"obius numbers are given in the following proposition from \cite[Chap. 3, Sect. 10]{stanley}. \begin{proposition} For a partition $P\in\Pi_n$, we have $\mu(P, \widehat 1_{\Pi_n})$ is given by $(-1)^{|P|-1}(|P|-1)!$, where $|P|$ is the number of parts of $P$.\end{proposition} So, we may write $P(\Pi_n, s)$ in the following form: $$P(\Pi_n, s) = \frac{1}{\binom n2^s} \sum_{k=1}^n \sum_{P\in\Pi_n;\text{ } |P| = k} (-1)^{k-1}(k-1)!\left(\sum_{Q\in P} \binom{\lvert Q\rvert}{2} \right)^s.$$

\section{COSET-LIKE BEHAVIOR}

Recall that an \textbf{ordinary Dirichlet series} is a series of the form $\sum_{n=1}^\infty a_n n^{-s}$, where $a_n\in\mathbb C$. This is a special case of a \textbf{general Dirichlet series}, which is a series of the form $\sum_{n=1}^\infty a_ne^{-\lambda_ns},$ where $a_n\in\mathbb C$ and $\lambda_n$ is a strictly increasing divergent sequence of non-negative real numbers.

Notice that in \eqref{4}, the ratio $[J : J_x]$ need not be an integer. Consequently, $P(L, s)$ need not be an ordinary Dirichlet series. In this regard, the obtained expression for $P(L, s)$ is a finite general Dirichlet series with integer coefficients. For example, we have $$P(\Pi_5, s) = 1 - \frac{5}{(5/3)^s} - \frac{10}{(5/2)^s} + \frac{20}{(10/3)^s} + \frac{6}{5^{s-1}} - \frac{6}{10^{s-1}}.$$

Since for a coset $xH$ in the coset lattice $\mathscr C(G)$ of a finite group $G$, the ratio $[J : J_{xH}] = |G:H|$ is always an integer (by Lagrange's theorem), we call lattices with this property coset-like. More precisely:

\begin{definition}
We say a lattice $L$ is \textbf{strongly coset-like} if $|J_x|$ divides $|J|$ for every $x\in L\setminus\{\widehat 0\}$. 
\end{definition}

This divisibility condition may not be necessary for $P(L, s)$ to be a finite ordinary Dirichlet series, for the M\"obius numbers may happen to cancel out in just the right way to eliminate non-integer ratios. Thus:

\begin{definition}
We say a lattice $L$ is \textbf{weakly coset-like} if $P(L, s)$ is a finite ordinary Dirichlet series.
\end{definition}

In principle, it is unclear whether the above two definitions are equivalent. Some structural property of lattices could potentially prevent ``local sums'' of the form $$S(q) := \sum_{\substack{x\in L \\ [J:J_x] = q}} \mu(x, \widehat 1)$$ from vanishing (note that $P(L, s) = \sum_{q\in\mathbb Q} S(q)q^{-s}$). However, this is not the case. A computer search by Håvard Damm-Johnsen \cite{Damm-Johnsen_WCLvsSCL} has revealed that the following lattice on 10 points (given by its Hasse diagram) is a minimal example of a weakly coset-like lattice which is not strongly coset-like.

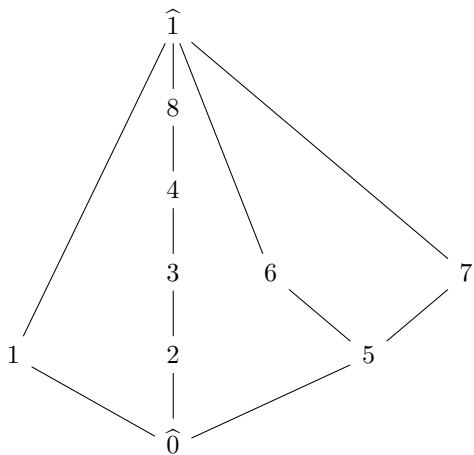
\adjustbox{scale=0.9,center}{\begin{tikzcd}
                          &  & \widehat 1                                                              &                         &                                           &                           \\
                          &  & 8 \arrow[u, no head]                                                    &                         &                                           &                           \\
                          &  & 4 \arrow[u, no head]                                                    &                         &                                           &                           \\
                          &  & 3 \arrow[u, no head]                                                    & 6 \arrow[luuu, no head] &                                           & 7 \arrow[llluuu, no head] \\
1 \arrow[rruuuu, no head] &  & 2 \arrow[u, no head]                                                    &                         & 5 \arrow[lu, no head] \arrow[ru, no head] &                           \\
                          &  & \widehat 0 \arrow[llu, no head] \arrow[u, no head] \arrow[rru, no head] &                         &                                           &                  \end{tikzcd}}

One readily computes that the above lattice has probabilistic zeta function $1-1/2^s - 2/4^s$, but $[J:J_3] = 8/3$. Indeed, one obtains an infinite family of such lattices by simply adjoining $k$ new atoms with $k$ chosen appropriately. This increases the number of join irreducibles to $8+k$ but does not alter the M\"obius numbers of the original lattice. Thus, whenever $k\not\equiv 1\pmod 3$ (so that $8+k$ is not divisible by $3$), we get an example of a weakly coset-like lattice on $10+k$ elements which is not strongly coset-like. This covers the residue classes $0$ and $1$ modulo 3. To get the residue class $2$ modulo 3, we start from the following lattice instead (also obtained by \cite{Damm-Johnsen_WCLvsSCL}).

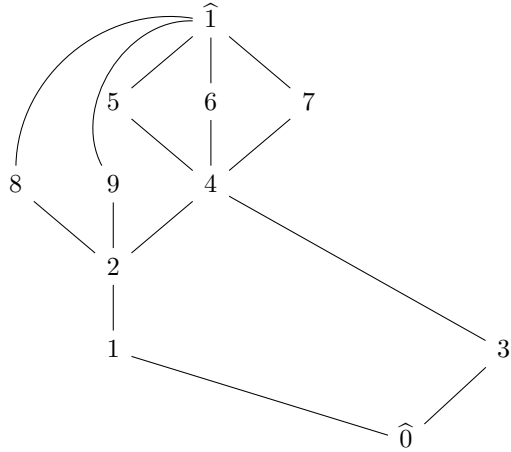
\adjustbox{scale=0.9,center}{\begin{tikzcd}
                                      &                                                              & \widehat 1                                                   &                       &                                                      &                          \\
                                      & 5 \arrow[ru, no head]                                        & 6 \arrow[u, no head]                                         & 7 \arrow[lu, no head] &                                                      &                          \\
8 \arrow[rruu, no head, bend left=49] & 9 \arrow[ruu, no head, bend left=60]                         & 4 \arrow[lu, no head] \arrow[u, no head] \arrow[ru, no head] &                       &                                                      &                          \\
                                      & 2 \arrow[lu, no head] \arrow[u, no head] \arrow[ru, no head] &                                                              &                       &                                                      &                          \\
                                      & 1 \arrow[u, no head]                                         &                                                              &                       &                                                      & 3 \arrow[llluu, no head] \\
                                      &                                                              &                                                              &                       & \widehat 0 \arrow[lllu, no head] \arrow[ru, no head] &            \end{tikzcd}}

This is also a weakly coset-like lattice which is not strongly coset-like, for its probabilistic zeta function is $1-3/2^s + 2/4^s$, yet $[J:J_4] = 8/3$. Following the same process of adding $k$ new atoms, we get an example of a weakly coset-like lattice on $11+k$ elements which is not strong coset-like whenever $k\not\equiv 1\pmod 3$ (such that $8+k$ is not divisible by 3). This covers the residue classes $1$ and $2$ modulo $3$. As such, we have examples of such lattices on any number of vertices above $10$. 

In contrast, the following lemma gives a simple criterion for when we may extend a ``not strongly coset-like'' result to a ``not weakly coset-like'' one.

\begin{lemma}
Let $L$ be a lattice with distinct identities $\widehat 0, \widehat 1$. If there exists $x\in L\setminus\{\widehat 0, \widehat 1\}$ such that $|J_x| \geq |J_y|$ for all $y\in L\setminus\{\widehat 0, \widehat 1\}$ and $|J_x|$ does not divide $|J|$ (i.e. the strongly coset-like condition for $L$ fails at $x$),
then $L$ is not weakly coset-like. \label{lemma}
\end{lemma}

\begin{proof}
Firstly, we show that the conditions of the lemma imply that $x$ is a coatom. If $x$ were not maximal, then there would exist $y\in L$ with $\widehat 1 > y > x$. Further, we would have $|J_y|\geq |J_x|\geq |J_y|$, so $|J_x| = |J_y|$. When combined with $J_x\subseteq J_y$, we would get $J_x = J_y$. Then, $x=\bigvee_{j\in J_x} j = \bigvee_{j\in J_y} j = y>x$, a contradiction. 

Since $x$ is a coatom, we have $\mu(x, \widehat 1) = -1$. Furthermore, if $y\in L$ is such that $|J_x|=|J_y|$, then $y$ is a coatom, thus also $\mu(y, \widehat 1) = -1$. It follows that the term $1/[J : J_x]^s$ appears in $P(L, s)$ with a non-zero coefficient, as desired to show that $P(L, s)$ is not an ordinary Dirichlet series. 
\end{proof}

\subsection{Partition Lattice}

The examples computed in the previous chapter show that the divisibility lattice $\mathcal O_n$ (hence also the Boolean lattice $\mathcal B_r$) and the subspace lattice $\mathcal S(\mathbb F_q^n)$ are typically not weakly coset-like, hence also not strongly coset-like. We now examine the partition lattice $\Pi_n$. As a first step, it is easy to prove that $\Pi_n$ is typically not strongly coset-like. This is expected, for we saw in the previous section that the structure of the join irreducibles in the partition lattice is additive in nature.

\begin{proposition}
The partition lattice $\Pi_n$ is strongly coset-like if and only if $n\leq 4$. \label{prop2}
\end{proposition}

\begin{proof}
It is trivial that $\Pi_2, \Pi_3$ are strongly coset-like, as there are no non-trivial elements which are not join-irreducible. For $\Pi_4$, it suffices to check that $\binom{3}{2}=3$ and $\binom{2}{2} + \binom{2}{2}=2$ divide $\binom{4}{2}=6$. 

Suppose now that $n\geq 5$. Consider the partition $P_0 = \{\{1\}, \{2, 3, 4, \ldots, n\}\}$. Then, $|J_{P_0}| = \binom{n-1}{2}$ does not divide $\binom{n}{2}=\binom{n-1}{2} + n-1$, for this is equivalent to $\binom{n-1}{2}\mid n-1$, yet $\binom{n-1}{2}>n-1$ for $n\geq 5$.
\end{proof}

Lemma \ref{lemma} allows us to deduce further that the partition lattice $\Pi_n$ is weakly coset-like if and only if it is strongly coset-like.

\begin{theorem}
The partition lattice $\Pi_n$ is weakly coset-like if and only if $n\leq 4$. 
\end{theorem}

\begin{proof}
By Proposition \ref{prop2}, it suffices to show that $P_0 = \{\{1\}, \{2, 3, \ldots, n\}\}\in\Pi_n$ satisfies the conditions of Lemma \ref{lemma} for $n\geq 5$.  Let $P = \{P_1, P_2 \ldots, P_k\}$ be any non-trivial (that is, $1<k<n$) partition in $\Pi_n$, set $l = \lfloor k/2\rfloor$ and consider $P' = \{P'_1, P'_2\}$, where $P'_1 = \sqcup_{i=1}^l P_i$ and $P'_2 = \sqcup_{i=l+1}^k P_i$. Thus, $P'$ is a coatom and by supra-additivity of the function $x\mapsto\binom x2$ (i.e. $\binom x2 + \binom y2 < \binom{x+y}2$ and induction), we get $$|J_P| = \sum_{i=1}^l \binom{|P_i|}2  + \sum_{i=l+1}^k \binom{|P_i|}2 \leq \binom{|P'_1|}2 + \binom{|P'_2|}2 \leq \binom{n-1}{2} = |J_{P_0}|.$$ The last inequality follows by the fact that the function $x\mapsto \binom{x}2 + \binom{n-x}2$ is symmetric with respect to $x=n/2$, strictly decreasing on $[1, n/2]$ and strictly increasing on $[n/2, n-1]$, so that its maximum on $[1, n-1]$ is attained at the boundary point $x=1$. Note that the inequality $|J_P|\leq\binom{n-1}2$ is strict when $P$ is not a coatom. 
\end{proof}

\subsection{$d$-Divisible Partition Lattice}

We turn our attention to the $d$-divisible partition lattice $\Pi_{dn}^d$, namely the set of all partitions of $\{1, 2, \ldots, dn\}$ with the property that each part is of cardinality divisible by $d$, ordered by refinement, with the empty set as an artificial bottom element. There was initially more hope for positive results regarding coset-like behavior of $\Pi_{dn}^d$, motivated by the multiplicative nature of the structure of the join irreducibles (see below). $d$-divisible partition lattices are also similar to coset lattices in terms of $EL$-labeling considerations \cite{woodroofe}. Although we unexpectedly obtained negative results, they are more intriguing than the previous results for $\Pi_n$. 

The join-irreducible elements of $\Pi_{dn}^d$ are precisely its atoms, i.e. partitions where all parts are of cardinality exactly equal to $d$. An elementary counting argument shows that for a partition $P$ with part sizes $dp_1, \ldots, dp_k$, we have $|J_P| = \prod_{i=1}^k (dp_i)!/((d!)^{p_i}p_i!)$, where $\sum_{i=1}^k p_i = n$. So, whether $\Pi_{dn}^d$ is a strongly coset-like lattice reduces to the following question about divisibility.

\begin{question}Let $n$ be a positive integer and let $k\in\{1, 2, \ldots, n\}$. Is it true that for any positive integers $p_1, p_2 \ldots, p_k$ such that $\sum_{i=1}^k p_i = n$, the divisibility relation $$\left(\prod_{i=1}^k\frac{(dp_i)!}{p_i!}\right) \Big\vert \left(\frac{(dn)!}{n!}\right)$$ holds?\end{question}

Notice that the multinomial $(\sum x_i)!/\prod x_i!$ being an integer implies both that the numerator of the product divides $(dn)!$ as well as that the denominator of the product divides $n!$. This naturally leads us to believe that the answer is no.

In order to show that $\Pi_{2\cdot 2m}^2$ is not strongly coset-like for any $m\geq 2$, it suffices to show that the partition $$P = \{\{1, 2 \ldots, 2m\}, \{2m+1, 2m+2, \ldots, 4m\}\}$$ fails to fulfill the strong coset-like property. That is, we would like to show that $((2m)!/m!)^2$ does not divide $(4m)!/(2m)!$, i.e. that $\binom{2m}{m}$ does not divide $\binom{4m}{2m}$ for any $m\geq 2$.

\begin{lemma}
Let $m\geq 2$ be an integer. Then, $\binom{2m}{m}$ does not divide $\binom{4m}{2m}$.
\end{lemma}

\begin{proof} Let $m\geq 15$. Notice that any prime $p\in (m, 2m)$ divides $(2m)!$ (and consequently $\binom{2m}{m}$) precisely once. Writing $\binom{4m}{2m} = (4m)!/(2m)!^2$, we would like to find a prime $p\in(m, 2m)$ which divides $(4m)!$ precisely twice as many times as $(2m)!$. That is, $p$ should satisfy $3p > 4m$, so that $p$ divides $(4m)!$ precisely twice.

By Theorem \ref{nagurathm}, there exists a prime $p_m\in [5m/3, 2m)$, therefore $p_m$ divides $\binom{2m}m$ but not $\binom{4m}{2m}$, as desired to show that $\binom{2m}m$ does not divide $\binom{4m}{2m}$. The remaining cases $(2\leq m\leq 14)$ follow by a computer check on \texttt{https://www.wolframalpha.com/}.\end{proof}

We remark that a trick similar to the above is used in the proof of Bertrand's postulate found in \cite[Chap. 2]{PfTB}. As a corollary, we obtain the following.

\begin{proposition}
The $2$-divisible partition lattice $\Pi_{2n}^2$ is strongly coset-like if and only if $n<4$ or $n=5$.
\end{proposition}

\begin{proof}
It is straightforward to check by hand that $\Pi_{2n}^2$ is strongly coset-like if $n<4$ or $n=5$. 

If $n\geq 4$ and $n$ is even, the previous lemma gives that $\Pi_{2n}^2$ is not strongly coset-like. Assume now that $n>5$ is odd and write it as $n=2m+1$ for some $m\geq 3$. Consider the partition $P = \{\{1, 2, \ldots, 2m \}, \{2m+1, 2m+2, \ldots, 4m+2\}\}\in\Pi_{2\cdot(2m+1)}^2$. Then, $|J_P|\mid |J|$ is equivalent to $(2m)!/m! \cdot (2m+2)!/(m+1)! \mid (4m+2)!/(2m+1)!$ or equivalently, \begin{equation}(2m+1)\binom{2m}{m}\text{ }\Big\vert\text{ } (4m+1)\binom{4m}{2m}.\label{eqnruss}\end{equation} Any prime $p_m\in[5m/3, 2m)$ lies strictly between $(4m+1)/3$ and $(4m+1)/2$, thus does not divide $4m+1$. Therefore, for $m\geq 15$, we are done by the proof of the previous lemma. The remaining cases follow by a computer check on  \texttt{https://www.wolframalpha.com/}.
\end{proof}

For $d>2$, we present an asymptotic result, for which we shall utilize the prime number theorem.

\begin{theorem}
For any integer $d\geq 2$, there exists $N = N(d)$ such that $\Pi_{dn}^d$ is not strongly coset-like for any $n\geq N$. \label{mainthm}
\end{theorem}

\begin{proof}
For $d=2$, we may take $N = 4$ (this is the content of the previous proposition). Now, assume that $d>2$ and firstly, let $n=2m$ for some $m$. Consider the partition $$P = \{\{1, 2, \ldots, dm\}, \{dm+1, dm+2, \ldots, 2dm\}\}\in\Pi_{2dm}^d.$$ The divisibility condition $|J_P|\mid |J|$ is equivalent to $$\left(\frac{(dm)!}{m!} \right)^2\text{ }\Big\vert\text{ } \frac{(2dm)!}{(2m)!}\text{, i.e. } \binom{2m}{m}\text{ }\Big\vert\text{ }\binom{2dm}{dm}.$$ We follow the same strategy as before: we would like to find a prime $p\in (m, 2m)$ (so that $p$ divides $\binom{2m}m$ precisely once) that does not divide $\binom{2dm}{dm}$. Let $\delta(d) = d/2$ if $d$ is even and $\delta(d) = (d+1)/2$ if $d$ is odd. Writing $\binom{2dm}{dm} = (2dm)!/(dm)!^2$, we would like to find a prime $p$ which divides $(dm)!$ with multiplicity $\delta(d)$ and divides $(2dm)!$ precisely twice as many times as $(dm)!$. Choose $N_1$ large enough so that primes $p>m\geq N_1/2$ satisfy $p^2 > 2dm$. Then, it suffices to find $p$ satisfying $$(2\delta(d) + 1/2)p > 2dm> 2\delta(d)p \text{ and } (\delta(d)+1/2)p>dm > \delta(d)p.$$ Rearranging gives $$\frac{2dm}{2\delta(d) + 1/2}< p < \frac{dm}{\delta(d)}.$$ Choosing $\varepsilon>0$ sufficiently small (i.e. $\varepsilon < \frac{2\delta(d) +1/2}{2\delta(d)} - 1$), we may, by Theorem \ref{pnt}, take $N(\varepsilon)\geq N_1$ such that there exists a prime $p_m$ in the desired interval for all $m\geq N(\varepsilon)$.

Suppose now that $n=2m+1$ for some $m$. Consider the partition $$Q = \{\{1, 2, \ldots, dm\}, \{dm+1, dm+2, \ldots, 2dm+d\}\}\in\Pi_{2dm+d}^d.$$ The divisibility relation $|J_Q|\mid|J|$ is equivalent to $$\frac{(dm)!}{m!}\cdot\frac{(dm+d)!}{(m+1)!}\text{ }\Big\vert\text{ }\frac{(2dm+d)!}{(2m+1)!}\text{, i.e. } \binom{2m}{m}\cdot\prod_{s=1}^{d-1} (dm+s)\text{ }\Big\vert\text{ }\binom{2dm}{dm}\cdot\prod_{s=1}^{d-1} (2dm+s).$$ It suffices to show that the same choice of $p_m$ as in the previous case does not divide the product $\prod_{s=1}^{d-1}(2dm+s)$. We have $$\frac{2dm+1}{\frac{dm}{\delta(d)}} < \frac{2dm+s}{p_m} < \frac{2dm+d-1}{\frac{2dm}{2\delta(d)+1/2}}$$ for every $s\in\{1, 2, \ldots, d-1\}$. As $m$ gets large, the left-hand side is arbitrarily close to and greater than $2\delta(d)$, whereas the right-hand side is arbitrarily close to $2\delta(d)+\frac{1}{2}$. This shows that for large enough $m$, $(2dm+s)/p_m$ is not an integer for any $s\in\{1, 2, \ldots, d-1\}$, thus $p_m$ does not divide $\prod_{s=1}^{d-1} (2dm+s)$. This concludes the proof.
\end{proof}

Note that unlike before for $\Pi_n$, Lemma \ref{lemma} cannot be used to extend Theorem \ref{mainthm} to a ``not weakly coset-like'' result, for $P_0 = \{\{1, 2, \ldots, dm\}, \{dm+1, dm+2, \ldots, 2dm\}\}\in\Pi_{2dm}^d$ is not where the maximum of $\lvert J_P\rvert$ is achieved. If we consider $$P_0' = \{\{1, 2, \ldots, d\}, \{d+1, d+2, \ldots, dn\}\}\in\Pi_{dn}^d$$ instead (because this is typically where the maximum is achieved), the corresponding divisibility condition is $$d!\cdot\frac{(d(n-1))!}{(n-1)!}\text{ }\Big\vert\text{ }\frac{(dn)!}{n!}.$$ This reduces to $(d-1)!\mid (dn-1)\cdots (dn-d+1)$. As the product of $d-1$ consecutive integers is divisible by $(d-1)!$, the latter divisibility does in fact hold.

\section{Products}

We recall that two groups $G$ and $H$ are \textbf{coprime} if no proper subgroup of $G\times H$ surjects onto both factors. Brown \cite{brown} proved that for coprime finite groups $G$ and $H$, the relation $P(G\times H, s) = P(G, s)P(H, s)$ holds. On the level of lattices, the identity is $$P(\mathscr C(G\times H), s+1) = P(\mathscr C(G), s+1)P(\mathscr C(H), s+1).$$ Thus, we are naturally led to ask: what is the appropriate product $\star$ on lattices for which $P(L\star K, s) = P(L, s)P(K, s)$?

We start with stating the following multiplicativity result involving the M\"obius numbers from \cite[Chap. 3, Sect. 8]{stanley}.

\begin{proposition}
Let $P$ and $Q$ be posets and let $\mu_P, \mu_Q$ and $\mu_{P\times Q}$ be the M\"obius functions of $P, Q$ and $P\times Q$, respectively. Then, for any $(x, y), (z, w)\in$ $P\times Q$, we have $$\mu_{P\times Q}((x, y), (z, w)) = \mu_P(x, z)\mu_Q(y, w).$$ In particular, if $P, Q$ are lattices, then $\mu_{P\times Q}((x, y), (\widehat 1, \widehat 1)) = \mu_P(x, \widehat 1)\mu_Q(y, \widehat 1)$. \label{mobius} \end{proposition} It is therefore natural to ask for the same multiplicative behavior in the join-irreducible ratios, i.e. that $|J_{(x, y)}| = |J_x|\cdot|J_y|$. It is immediate that the Cartesian product does not work, since the element $(x, y)$ need not be join-irreducible if $x, y$ are, for $(x, y)$ may be written as a non-trivial join as $(x, \widehat 0)\vee (\widehat 0, y)$.

\subsection*{Lower Reduced Product}

We consider the lower reduced product $\star$ of lattices, defined as $L\star K := (L\setminus\{\widehat 0\})\times(K\setminus\{\widehat 0\})\cup\{\widehat 0\}$, with componentwise meet and join. This is a natural product to consider, because for groups $G, H$ for which all subgroups of $G\times H$ may be factored as $S\times T$ for some $S\leq G$ and some $T\leq H$, the coset lattices satisfy $\mathscr C(G\times H) = \mathscr C(G)\star\mathscr C(H)$. This happens, for instance, when $G$ and $H$ have coprime orders (see \cite[Chap. 2, Sect. 4]{suzuki}). It is worth noting that Brown \cite{brown} made a more general related observation involving homotopy equivalence. Adding a mild assumption for one of the lattices, we get the following result.

\begin{proposition}
Let $L$ and $K$ be lattices, where $L$ is an atomistic lattice. Then, we have $$P(L\star K, s) = P(L, s)P(K, s).$$
\end{proposition}

\begin{proof}
By Proposition \ref{mobius}, the M\"obius function remains multiplicative over the lower reduced product. That is, $\mu((x, y), (\widehat 1, \widehat 1)) = \mu(x, \widehat 1)\mu(y, \widehat 1)$, as in $P(L\star K, s)$, the M\"obius numbers are taken over the poset $(L\star K)\setminus\{\widehat 0\} = (L\setminus\{\widehat 0\})\times(K\setminus\{\widehat 0\})$, a Cartesian product of posets.

Next, we show that $J_{(x, y)} = J_x\times J_y$. If $(x, y)\in L\star K$ is join-irreducible, then for $x = a\vee c$ and $y = b\vee d$, we have $(x, y) = (a, b)\vee(c, d)$. This means that $(a, b) = (x, y)$ or $(c, d) = (x, y)$, as desired to show that both $x, y$ are join-irreducible. 

Conversely, if $x$ is join-irreducible in $L$ and $y$ is join-irreducible in $K$, then for $$(x, y) = (a, b)\vee (c, d) = (a\vee c, b\vee d),$$ since $x$ is an atom, without loss of generality, we may take $a=x$. If $c=\widehat 0$, this forces $d=\widehat 0$ due to the definition of $\star$, so $b=y$, i.e. $(a, b) = (x, y)$. So, assume that $c=x$ also. If $b\neq y$, join-irreducibility of $y$ implies that $d=y$, so that $(c, d) = (x, y)$, as desired to show that $(x, y)$ is join-irreducible. We finish the proof with a simple computation. We have 

\begin{eqnarray*}
P(L\star K, s) & = & \sum_{\widehat 0 < (x, y)\in L\star K} \frac{\mu((x, y), (\widehat 1,\widehat 1))}{[J(L\star K):J_{(x, y)}]^s} \\ & = & \sum_{\widehat 0 < x\in L}\sum_{\widehat 0 < y\in K} \frac{\mu(x, \widehat 1)\cdot\mu(y,\widehat 1)}{[J(L):J_{x}]^s\cdot[J(K) : J_y]^s} \\ & = & \left(\sum_{\widehat 0 < x\in L} \frac{\mu(x, \widehat 1)}{[J(L) : J_x]^s} \right)\left(\sum_{\widehat 0 < y\in K} \frac{\mu(y, \widehat 1)}{[J(K) : J_y]^s} \right) \\ \\ & = & P(L, s)P(K, s),
\end{eqnarray*} as we desired to show. 
\end{proof}

\section{Questions} We now present some problems that remain open and are possible interesting avenues for further research.
    
    One might be interested in obtaining an alternative (e.g. purely lattice-theoretic) characterization of coset-like lattices. The most naive question is whether a lattice is (strongly) coset-like if and only if it is a coset lattice. In this regard, note that the partititon lattices $\Pi_2$ and $\Pi_3$ are both isomorphic to coset lattices, namely the coset lattices of the trivial group and $\mathbb Z/3\mathbb Z$, respectively. However, $\Pi_4$ is not isomorphic to the coset lattice of any group. For if $\Pi_4$ would be isomorphic to $\mathscr C(G)$ for some group $G$, the group would be of order $6$ because $\mathscr C(G)$ has precisely $|G|$ join-irreducible elements and $\Pi_4$ has $6$ join-irreducible elements. So, we would either have $G\cong \mathbb Z/6\mathbb Z$ or $G\cong S_3$ (the symmetric group on 3 points) but $\mathscr C(\mathbb Z/6\mathbb Z)$ consists of $13$ elements and $\mathscr C(S_3)$ consists of $19$ elements. In contrast, $\Pi_4$ has $15$ elements, so it is not isomorphic to either. A less naive but vaguer question is then: 
    
    \begin{question}
    May every (strongly) coset-like lattice be embedded in a coset lattice in a way that is compatible with the ingredients of the definition of the probabilistic zeta function of a lattice? 
    \end{question}
    
    Note that merely being a sublattice is not instructive, for the subgroup lattice of every group is a sublattice of the coset lattice, yet we saw that the subgroup lattice of $\mathbb Z/n\mathbb Z$ (i.e. the divisibility lattice $\mathcal O_n$) is certainly not strongly coset-like. Because the coset lattice $\mathscr C(G)$ of a group has a natural $G$-action associated with it given by $(xH)^g = (gx)H$, one might expect better behavior from sublattices which preserve the group action through a subgroup. More precisely, if $L$ is a lattice and $G$ is a group which acts on $L$, we may consider sublattices $K$ of $L$ for which there exists some subgroup $H\leq G$ which acts on $K$ via the restricted action. It turns out, however, that the structure of join-irreducible elements of such sublattices may still significantly differ from the structure of the original lattice. For example, the sublattice of $\mathscr C(\mathbb Z/6\mathbb Z)$ generated by $\{0\}, \{3\}$ and $\{1, 4\}$ is not a weakly coset-like lattice, although the subgroup $\{0, 3\}\leq\mathbb Z/6\mathbb Z$ acts on it via the restricted action of $\mathbb Z/6\mathbb Z$ on $\mathscr C(\mathbb Z/6\mathbb Z)$. Nonetheless, we regard compatibility with the group action as natural and therefore require it. Additionally, we require that the structure of join-irreducible elements is preserved when passing to a sublattice. One possible model of appropriate sublattices is the following. \begin{definition}
Let $G$ be a finite group. We say a sublattice $L$ of the coset lattice $\mathscr C(G)$ is a \textbf{good} sublattice, if all of the following hold: \begin{itemize}
    \item[(i)] there exists a normal subgroup $H\trianglelefteq G$ such that $H$ acts on $L$ via the restricted action of $G$ on $\mathscr C(G)$,
    
    \item[(ii)] the set $J(L)$ of join-irreducible elements of $L$ is a subset of the set of join-irreducible elements of $\mathscr C(G)$ (and therefore corresponds to a subset of $G$), and
    
    \item[(iii)] no three elements in $J(L)$ correspond to group elements that belong to three distinct cosets of $H$ in $G$. 
\end{itemize}
\end{definition}

Note that whole coset lattice is trivially a good sublattice of itself. We show that the above definition guarantees that a good sublattice of $\mathscr C(G)$ is strongly coset-like.

\begin{proposition}
Let $G$ be a finite group. Then, any good sublattice of $\mathscr C(G)$ is strongly coset-like.
\end{proposition}

\begin{proof}
Let $L$ be a good sublattice of $\mathscr C(G)$ and let $H\trianglelefteq G$ be the normal subgroup that acts on it. Notice that the first clause of the definition of a good sublattice implies that if some coset $xK$ of some subgroup $K$ of $G$ is in $L$, then so is $(hx)K$ for every $h\in H$. In particular, if $\{g\}\in J(L)$ for some $g\in G$, then $\{hg\}\in J(L)$ for all $h\in H$.

Since $L$ is a good sublattice, $J(L)$ is the union of at most two cosets of $H$, say, $Hg_1 = g_1H$ and $Hg_2 = g_2H$ (note the use of normality of $H$). For any $xK\in L$, the set $J_{xK}^L$ of join-irreducible elements in $L$ less than or equal to $xK$ has cardinality $$\lvert J_{xK}^L \rvert = \begin{cases} \lvert xK\cap g_1H\rvert & \text{if } g_1H = g_2H,\\ \lvert xK \cap g_1H\rvert + \lvert xK \cap g_2H\rvert & \text{if } g_1H\neq g_2H.\end{cases}$$ Note that in general, $xK\cap gH$ is either empty or a coset of $K\cap H$. Thus, $\lvert J_{xK}^L\rvert$ is either equal to $\lvert K\cap H\rvert$ or $2\lvert K\cap H\rvert$ (in the second case, this depends on whether only one or both of the intersections are non-empty). If $g_1H\neq g_2H$, because $K\cap H\leq H$ and $\lvert J(L)\rvert = 2\lvert H\rvert$, it follows that $\lvert J_{xK}^L \rvert \mid \lvert J(L)\rvert$ for any $xK\in L$. The same conclusion holds if $g_1H = g_2H$, hence $L$ is strongly coset-like. \end{proof}

We remark that for $\mathscr C(\mathbb Z/8\mathbb Z)$ and $H = \{0, 4\}\leq\mathbb Z/8\mathbb Z$, the sublattice $L\leq\mathscr C(\mathbb Z/8\mathbb Z)$ generated by $J(L) = \{\{0\}, \{1\}, \{2\}, \{4\}, \{5\}, \{6\}\}$ (i.e. such that $J(L)$ consists of $3$ distinct cosets of $H$) is acted on by $H$, yet is not strongly coset-like. This plausibly justifies clause (iii) in the definition of a good sublattice, although this phenomenon is likely part of a more general scheme. As for the normality assumption in clause (i), it serves to control the tension between left and right cosets.

If $L\leq\mathscr C(G)$ is a good sublattice with $H\trianglelefteq G$ acting on it and $J(L)$ corresponds to precisely one coset $Hx$ of $H$, then it is clear that $L\cong\mathscr C(H)$ (by the isomorphism that simply adjoins $x$ everywhere). Therefore, by the first remark of this subsection, we observe that $\Pi_4$ cannot be embedded in a coset lattice as a good sublattice with the set of join-irreducible elements corresponding to a single coset of some normal subgroup. However, it remains an open question whether this can be done using a normal subgroup $H$ of order $3$ of some group $G$ and 2 distinct cosets of $H$.

We close with the following question, which also remains open. All examples of weakly coset-like lattices that are not strongly coset-like in this paper are non-atomistic lattices. A computer search \cite{Damm-Johnsen_WCLvsSCL} up to lattices on 10 points (including a significant portion of lattices on 11 points) failed to find any such atomistic examples. Therefore:  

    \begin{question} Does there exist a weakly coset-like atomistic lattice that is not strongly coset-like? In particular, is the $d$-divisible partition lattice weakly coset-like? \end{question}
    

\bibliographystyle{amsplain}
\bibliography{references}

\end{document}